\documentclass[12pt,a4paper]{article}
\usepackage{latexsym,amsthm,amssymb,amscd,amsmath}
\usepackage[colorlinks]{hyperref}
\usepackage[left=2.2cm,bottom=2.7cm,top=2.7cm,right=2.2cm]{geometry}
\usepackage[all]{xy}
\newcounter{alphthm}
\usepackage{float}
\newtheorem{thm}{Theorem}[section]
\newtheorem{lem}[thm]{Lemma}
\newtheorem{cor}{Corollary}[section]

\newtheorem{con}{Conjecture}[section]
\theoremstyle{definition}

\newcommand{\be}{\begin{equation}}
\newcommand{\ee}{\end{equation}}
\newcommand{\pa}{{\partial}}

\newcommand{\g}{{\bf g}}

\title{On Homogeneous Landsberg  Surfaces}
\author{A. Tayebi and B. Najafi }
\numberwithin{equation}{section}
\begin{document}
\maketitle

\begin{abstract}
In this paper, we prove that every homogeneous Landsberg surface has isotropic  flag curvature. Using this special form of the flag curvature,   we prove a rigidity result on homogeneous Landsberg surface. Indeed, we prove  that every homogeneous Landsberg surface is Riemannian or locally Minkowskian. This gives a positive answer to the Xu-Deng's well-known conjecture in 2-dimensional homogeneous Finsler manifolds.
\end{abstract}
\section{Introduction}
Let  $(M, F)$ be  a Finsler manifold and $c: [a, b]\rightarrow M$ be a piecewise $C^\infty$ curve from $c(a)=p$ to $c(b)=q$. For every $u\in T_pM$, let us define $P_c:T_pM\rightarrow T_qM$ by $P_c(u):=U(b)$, where $U=U(t)$ is the parallel vector field along $c$ such that
$U(a)=u$. $P_c$ is called the parallel translation along $c$. In \cite{I},  Ichijy\={o} showed that if  $F$ is a Berwald metric, then  all tangent
spaces $(T_xM, F_x)$ are linearly isometric to each other. Let us consider the Riemannian metric ${\hat g}_x$ on $T_xM_0:=T_xM-\{0\}$ which is defined by
${\hat g}_x:=g_{ij}(x, y)\delta y^i\otimes \delta y^j$, where $g_{ij}:={1}/{2}[F^2]_{y^iy^j}$ is the fundamental tensor of $F$ and $\{\delta y^i:= dy^i+N^i_j dx^j\}$ is the natural coframe on $T_xM$ associated with the natural basis $\{{\partial}/{\partial x^i}|_x\}$ for $T_xM$. If $F$ is a Landsberg metric, then for any $C^\infty$ curve $c$,  $P_c$  preserves the induced Riemannian metrics on the tangent spaces, i.e., $P_c:(T_pM, {\hat g}_p)\rightarrow (T_qM, {\hat g}_q)$ is an isometry. By definition, every Berwald metric is a Landsberg metric, but the converse may not hold.

In 1996, Matsumoto found a list of rigidity results which almost suggest that such a pure  Landsberg metric (non-Berwaldian metric) does not exist \cite{Mat96}. In 2003, Matsumoto emphasized this problem again and looked at it as the most important open problem in Finsler geometry.
It is a long-existing open problem in Finsler geometry to find Landsberg metrics which are not Berwaldian. Bao called such metrics unicorns in
Finsler geometry, mythical single-horned horse-like creatures that exist in legend but have never been seen by human beings \cite{Bao2}. There are a lot of unsuccessful attempts
to find explicit examples of  unicorns.   In \cite{Szabo2}, Szab\'{o} made an argument to prove that any regular Landsberg metric must be of Berwald type. But unfortunately, there is a little gap in Szab\'{o}'s argument. As pointed out in Szab\'{o}'s correction to \cite{Szabo2}, his argument only applies to the so-called dual Landsberg spaces.  Hence, the unicorn problem remains open in Finsler geometry. Taking into account of so many unsuccessful efforts of
 many researchers, one can conclude  that unicorn problem is becoming more and more puzzling.

\newpage

The unicorn problem  in Finsler geometry is well-studied. However, up to now, very little attention has been paid to the subject of homogeneous Finsler metrics. A Finsler manifold $(M, F)$ is said to be homogeneous if  its group of isometries  acts transitively on $M$. In \cite{TN1}, the authors consider the unicorn problem in the class of  homogeneous $(\alpha, \beta)$-metric. We proved that every homogeneous $(\alpha, \beta)$-metric is a stretch metric if and only if it is a  Berwald metric.  In \cite{XD}, Xu-Deng  introduced a generalization of  $(\alpha,\beta)$-metrics,
 called $(\alpha_1,\alpha_2)$-metrics. Let $(M, \alpha)$ be an $n$-dimensional Riemannian manifold. Then one can define an $\alpha$-orthogonal decomposition of the tangent bundle  by $TM=\mathcal{V}_1\oplus\mathcal{V}_2$,
where $\mathcal{V}_1$ and $\mathcal{V}_2$ are two   linear subbundles with dimensions $n_1$ and $n_2$ respectively, and
$\alpha_i=\alpha|_{\mathcal{V}_i}$ $i=1,2$  are naturally viewed as  functions on $TM$.
An  $(\alpha_1,\alpha_2)$-metric on $M$ is a  Finsler metric $F$ which can be written as
$F=\sqrt{L(\alpha_1^2,\alpha_2^2)}$. An $(\alpha_1,\alpha_2)$-metric can also be represented as
$F=\alpha\phi(\alpha_2/\alpha)=\alpha\psi(\alpha_1/\alpha)$, in which
$\phi(s)=\psi(\sqrt{1-s^2})$. They proved that  evey Landsberg $(\alpha_1,\alpha_2)$-metric reduces to a Berwald metric. This result shows that the finding  a unicorn cannot be successful even in the very broad class of $(\alpha_1,\alpha_2)$-metrics. Then, Xu-Deng conjectured the following:
\begin{con}{\rm (\cite{XD})}
A  homogeneous Landsberg space must be a Berwald space.
\end{con}

Taking a look at the rigid theorems in Finsler geometry, one can find that this type of result is different for procedures with dimensions greater than three. For example, in \cite{Sz} Szab\'{o} proved that any connected Berwald surface  is  locally Minkowskian or Riemannian. In  \cite{BCS}, Bao-Chern-Shen  proved a rigidity result for compact Landsberg surface. They showed that a compact Landsberg surfaces with non-positive flag curvature is  locally Minkowskian or Riemannian.  Therefore, we preferred to consider the issue of unicorns for homogeneous Finsler surfaces. In this paper, we prove the following rigidity result.
\begin{thm}\label{MainTHM1}
Any homogeneous Landsberg surface of is Riemannian or locally Minkowskian.
\end{thm}
This result articulates  the hunters of unicorns that they do not looking forward to seeing such a creature in the jungle of homogeneous Finsler surfaces.
\bigskip

In order to prove Theorem \ref{MainTHM1}, we consider the  flag curvature of Landsberg surface and prove  the following rigidity result.
\begin{thm}\label{MainTHM2}
Every homogeneous Landsberg surface has isotropic  flag curvature.
\end{thm}

\section{Preliminaries}\label{sectionP}
Let $(M, F)$ be an $n$-dimensional Finsler manifold, and $TM$ be its tangent space. We denote the slit tangent space of $M$ by $TM_0$, i.e., $T_xM_0=T_xM-\{0\}$ at every $x\in M$.  The fundamental tensor $\textbf{g}_y:T_xM\times
T_xM\rightarrow \mathbb{R}$ of $F$ is defined by following
\[
\textbf{g}_{y}(u,v):={1 \over 2}\frac{\pa ^2}{\pa s \pa t} \Big[ F^2 (y+su+tv)\Big]|_{s,t=0}, \ \
u,v\in T_xM.
\]
Let $x\in M$ and $F_x:=F|_{T_xM}$. To measure the
non-Euclidean feature of $F_x$, define ${\bf C}_y:T_xM\times T_xM\times
T_xM\rightarrow \mathbb{R}$ by
\[
{\bf C}_{y}(u,v,w):={1 \over 2} \frac{d}{dt}\Big[\textbf{g}_{y+tw}(u,v)
\Big]|_{t=0}, \ \ u,v,w\in T_xM.
\]
The family ${\bf C}:=\{{\bf C}_y\}_{y\in TM_0}$ is called the Cartan torsion. By definition, ${\bf C}_y$ is a symmetric trilinear form on $T_xM$. It is well known that ${\bf{C}}=0$ if and only if $F$ is Riemannian.

Let $(M, F)$ be a Finsler manifold. For  $y\in T_x M_0$, define ${\bf I}_y:T_xM\rightarrow \mathbb{R}$
by
\[
{\bf I}_y(u)=\sum^n_{i=1}g^{ij}(y) {\bf C}_y(u, \partial_i, \partial_j),
\]
where $\{\partial_i\}$ is a basis for $T_xM$ at $x\in M$. The family
${\bf I}:=\{{\bf I}_y\}_{y\in TM_0}$ is called the mean Cartan torsion. By definition, ${\bf I}_y(u):=I_i(y)u^i$, where $I_i:=g^{jk}C_{ijk}$. By Deicke's theorem, every positive-definite Finsler metric
 $F$ is Riemannian if and only if ${\bf I}=0$.

\bigskip

Given a Finsler manifold $(M, F)$, then a global vector field ${\bf G}$ is induced by $F$ on $TM_0$, and in a standard coordinate $(x^i,y^i)$ for $TM_0$ is given by ${\bf G}=y^i {{\partial} / {\partial x^i}}-2G^i(x,y){{\partial}/ {\partial y^i}}$, where $G^i=G^i(x, y)$ are scalar functions on $TM_0$ given by
\[
G^i:=\frac{1}{4}g^{ij}\Bigg\{\frac{\partial^2[F^2]}{\partial x^k
\partial y^j}y^k-\frac{\partial[F^2]}{\partial x^j}\Bigg\},\ \
y\in T_xM.\label{G}
\]
The vector field ${\bf G}$ is called the spray associated with $(M, F)$.

\bigskip

For $y \in T_xM_0$, define ${\bf B}_y:T_xM\times T_xM \times T_xM\rightarrow T_xM$ by ${\bf B}_y(u, v, w):=B^i_{\ jkl}(y)u^jv^kw^l{{\partial } \over {\partial x^i}}|_x$ where
\[
B^i_{\ jkl}:={{\partial^3 G^i} \over {\partial y^j \partial y^k \partial y^l}}.
\]
The quantity $\bf B$ is called the Berwald curvature of the Finsler metric $F$. We call a Finsler metric $F$ a Berwald metric,  if ${\bf{B}}=0$. 

\bigskip
Define the mean of Berwald curvature by ${\bf E}_y:T_xM\times T_xM \rightarrow \mathbb{R}$, where
\[
{\bf E}_y (u, v) := {1\over 2} \sum_{i=1}^n g^{ij}(y) g_y \Big ( {\bf B}_y (u, v, e_i ) , e_j \Big ).
\]
The family ${\bf E}=\{ {\bf E}_y \}_{y\in TM\setminus\{0\}}$ is called the {\it mean Berwald curvature} or {\it E-curvature}.
In a local coordinates,  ${\bf E}_y(u, v):=E_{ij}(y)u^iv^j$, where
\[
E_{ij}:=\frac{1}{2}B^m_{\ mij}.
\]
The quantity  $\bf E$ is called the mean
Berwald curvature. $F$ is called a weakly Berwald metric if ${\bf{E}}=0$.  Also, define  ${\bf H}_y:T_xM\otimes T_xM \rightarrow \mathbb{R}$ by ${\bf H}_y(u,v):=H_{ij}(y)u^iv^j$, where
\[
H_{ij}:= E_{ij|s} y^s.
\]
Then  ${\bf H}_y$ is defined as the covariant derivative of ${\bf E}$ along geodesics.

\bigskip
For non-zero vector $y \in T_xM_0$, define   ${\bf D}_y:T_xM\otimes T_xM \otimes T_xM\rightarrow T_xM$
 by  ${\bf D}_y(u,v,w):=D^i_{\ jkl}(y)u^iv^jw^k\frac{\partial}{\partial x^i}|_{x}$, where
\[
D^i_{\ jkl}:=\frac{\partial^3}{\partial y^j\partial y^k\partial y^l}\Bigg[G^i-\frac{2}{n+1}\frac{\partial G^m}{\partial y^m} y^i\Bigg].\label{Douglas1}
\]
$\bf D$ is called the Douglas curvature.  $F$ is called a Douglas metric if $\bf{D}=0$.  According to the definition, the Douglas tensor can be written as follows
\[
D^i_{\ jkl}=B^i_{\ jkl}-\frac{2}{n+1}\Big\{E_{jk}\delta^i_{\ l}+E_{kl}\delta^i_{\ j}+E_{lj}\delta^i_{\ k}+E_{jk,l}y^i\Big\}.
\]

\bigskip

For $y\in T_xM$, define the Landsberg curvature ${\bf L}_y:T_xM\times T_xM \times T_xM\rightarrow \mathbb{R}$ by
\[
{\bf L}_y(u, v,w):=-\frac{1}{2}{\bf g}_y\big({\bf B}_y(u, v, w), y\big).
\]
$F$ is called a Landsberg metric if ${\bf L}_y=0$. By definition, every Berwald metric is a Landsberg metric.

\bigskip
Let $(M, F)$ be a Finsler manifold. For  $y\in T_x M_0$, define ${\bf J}_y:T_xM\rightarrow \mathbb{R}$
by
\[
{\bf J}_y(u)=\sum^n_{i=1}g^{ij}(y) {\bf L}_y(u, \partial_i, \partial_j).
\]
The quntity $\bf J$ is called the mean Landsberg curvature or J-curvature of Finsler metric $F$.
A Finsler metric $F$ is called a weakly Landsberg metric if ${\bf J}_y=0$. By definition, every Landsberg metric is a weakly Landsberg metric.
Mean Landsberg curvature can also be defined as following
\[
J_i: = y^m {\pa I_i \over \pa x^m} -I_m {\pa G^m\over \pa y^i} - 2 G^m {\pa I_i \over \pa y^m}.
\]
By definition, we get
\begin{eqnarray*}
{\bf J}_y (u):= {d\over dt} \Big [ {\bf I}_{\dot{\sigma}(t) } \big ( U(t) \big )\Big ]_{t=0},
\end{eqnarray*}
where $y\in T_xM$, $\sigma=\sigma(t)$ is the geodesic with $\sigma(0)=x$, $\dot{\sigma}(0)=y$, and $U(t)$ is a  linearly parallel vector field along $\sigma$ with
$U(0)=u$. The mean Landsberg curvature ${\bf J}_y$ is the rate of change of ${\bf I}_y$ along geodesics
for any $y\in T_xM_0$.

\bigskip
For an arbitrary  non-zero vector $y \in T_xM_{0}$, the Riemann curvature is a  linear
transformation $\textbf{R}_y: T_xM \rightarrow T_xM$ with homogeneity ${\bf R}_{\lambda y}=\lambda^2 {\bf R}_y$,
$\forall \lambda>0$, which is defined by
$\textbf{R}_y(u):=R^i_{k}(y)u^k {\partial / {\partial x^i}}$, where
\be
R^i_{k}(y)=2{\partial G^i \over {\partial x^k}}-{\partial^2 G^i \over
{{\partial x^j}{\partial y^k}}}y^j+2G^j{\partial^2 G^i \over
{{\partial y^j}{\partial y^k}}}-{\partial G^i \over {\partial
y^j}}{\partial G^j \over {\partial y^k}}.\label{Riemannx}
\ee
The family $\textbf{R}:=\{\textbf{R}_y\}_{y\in TM_0}$ is called the Riemann curvature of the Finsler manifold $(M, F)$.

\bigskip

For a flag $P:={\rm span}\{y, u\} \subset T_xM$ with flagpole $y$, the flag curvature ${\bf
K}={\bf K}(x, y, P)$ is defined by
\be
{\bf K}(x, y, P):= {\g_y \big(u, {\bf R}_y(u)\big) \over \g_y(y, y) \g_y(u,u)
-\g_y(y, u)^2 }.\label{FC0}
\ee
The flag curvature ${\bf K}(x, y, P)$ is a function of tangent planes $P={\rm span}\{ y, v\}\subset T_xM$. This quantity tells us how curved space is at a point. A Finsler metric $F$ is of scalar flag curvature if $\textbf{K} = \textbf{K}(x, y, P)$ is independent of flags $P$ containing $y\in T_xM_0$.

\section{Proof of Theorems}
In  this section, we are going to prove Theorems \ref{MainTHM1} and \ref{MainTHM2}. In order to prove Theorem \ref{MainTHM1},  first we consider the flag curvature of homogeneous Landsberg surface. More precisely, we prove Theorem \ref{MainTHM2}.  For this aim, we need some useful Lemmas as follows.

In \cite{LR}, Latifi-Razavi proved that every  homogeneous Finsler manifold is forward geodesically complete. In \cite{TN1}, Tayebi-Najafi improved their result and proved the following.
\begin{lem}{\rm (\cite{TN2})}\label{lem1}
Every homogeneous Finsler manifold is complete.
\end{lem}

\bigskip
By definition, every two points of a homogeneous Finsler manifold $(M, F)$ map to each other under an isometry. This causes the norm of  an invariant  tensor  under the isometries of  a homogeneous Finsler manifold is a constant function on $M$, and consequently, it has a bounded norm. Then, we conclude the following.
\begin{lem}{\rm (\cite{TN1})}\label{lem2}
Let $(M, F)$ be a homogeneous Finsler manifold. Then, every invariant tensor under the isometries of $F$ has a bounded norm with respect to $F$.
\end{lem}

\bigskip

\noindent
{\bf Proof of Theorem \ref{MainTHM2}:} We first deal with Finsler surfaces. The special and useful Berwald frame was introduced and developed by Berwald \cite{B}. Let $(M, F)$ be a two-dimensional Finsler manifold. One can define a local field of orthonormal frame $(\ell^i,m^i)$ called the Berwald frame, where  $\ell^i=y^i/F$, $m^i$ is the unit vector with $\ell_i m^i=0$,   $\ell_i=g_{ij}\ell^i$ and $g_{ij}$ is  defined by $g_{ij}=\ell_i\ell_j+m_im_j$.  In \cite{BM}, it is proved that the Douglas curvature  of the Finsler surface $(M, F)$ is given by following
\[
D^i_{\ jkl}  = -\frac{1}{ 3F^2}\Big(6I_{,1}+ I_{2|2} + 2II_2\Big)m_jm_km_ly^i.
\]
We rewrite it as   equivalently
\be
{\bf D}_y(u,v,w)={\bf T}(u, v, w) y\label{GDW1}
\ee
where ${\bf T}(u, v, w):=T_{ijk}u^iv^jw^k$ and $T_{ijk}:=-{1}/(3F^2)(6I_{,1}+ I_{2|2} + 2II_2)m_im_jm_k$.  It is easy to see that ${\bf T}$ is a symmetric Finslerian tensor filed and satisfies the following
\[
{\bf T}(y, v, w)=0.
\]
Let us denote the Berwald connection  of $F$ by $D$. The horizontal and vertical derivation with of a Finsler tensor field are denoted by `` $D_{u}$ " and `` $D_{\dot u}$ " respectively. Taking a horizontal derivation of \eqref{GDW1} along Finslerian geodesics implies that
\be
D_0{\bf D}_y(u,v,w)=D_0{\bf T}(u, v, w)y,\label{GDW1.5}
\ee
where $D_0:=D_iy^i$. Let us define ${\bf h}_y:T_xM\to T_xM$ by
\[
{\bf h}_y(u)=u-{1 \over F^2 }{\bf g}_y(u,y)y.
\]
Since ${\bf h}_y(y)=0$, it follows from  \eqref{GDW1.5}  that
\begin{equation}
{\bf h}_y\big(D_0{\bf D}_y(u,v,w)\big)=0.\label{GDW2}
\end{equation}
On the other hand. the Douglas tensor of $F$  is given by
\begin{equation}
{\bf D}_y(u,v,w)={\bf B}_y(u,v,w)-\frac{2}{3}\Big\{{\bf E}_y(v, w)u+{\bf E}_y(w, u)v+{\bf E}_y(u, v)w+(D_{\dot u}{\bf E}_y)(v,w)y)\Big\}.\label{GD2}
\end{equation}
Then
\begin{equation}
{\bf h}_y\big(D_0{\bf D}_y(u, v, w)\big)={\bf h}_y\big(D_0{\bf B}_y(u, v, w)\big)-\frac{2}{3}\Big\{{\bf H}_y(u,v) {\bf h}_y(w)+{\bf H}_y(v, w) {\bf h}_y(u)+{\bf H}_y(w, u) {\bf h}_y(v)\Big\}.\label{GD3}
\end{equation}
Let us define
\[
\tilde{\bf B}_y:=D_0{\bf B}_y.
\]
Indeed, $\tilde{\bf B}_y$ is the horizontal derivative of Berwald curvature along Finsler geodesics. By (\ref{GDW2}) and (\ref{GD2}), we get
\begin{equation}
{\bf h}_y\big(\tilde{\bf B}_y(u, v, w)\big)=\frac{2}{3}\Big\{{\bf H}_y(u,v) {\bf h}_y(w)+{\bf H}_y(v, w) {\bf h}_y(u)+{\bf H}_y(w, u) {\bf h}_y(v)\Big\}.\label{GD4a}
\end{equation}
Using $D_i{\bf h}=0$ yields
\begin{equation}
{\bf h}_y\big(D_i\tilde{\bf B}_y(u, v, w)\big)=\frac{2}{3}\Big\{D_i{\bf H}_y(u,v) {\bf h}_y(w)+D_i{\bf H}_y(v, w) {\bf h}_y(u)+D_i{\bf H}_y(w, u) {\bf h}_y(v)\Big\}.\label{GD5}
\end{equation}
Using $g_y({\bf B}_y(u,v,w),y)=-2{\bf L}_y(u,v,w)$, we get
\begin{eqnarray}\label{GD7}
D_i\big( {\bf h}_y\tilde{\bf B}_y(u, v, w)\big)\!\!\!\!&=&\!\!\!\! {\bf h}_y\big(D_i\tilde{\bf B}_y(u, v, w)\big)\nonumber\\
\!\!\!\!&=&\!\!\!\!  D_iD_0\big( {\bf h}_y{\bf B}_y(u, v, w)\big)\nonumber\\
\!\!\!\!&=&\!\!\!\!  D_iD_0\Big ({\bf B}_y(u, v, w)-{1 \over F^2}{\bf g}_y\big({\bf B}_y(u, v, w), y\big)\Big )\nonumber\\
\!\!\!\!&=&\!\!\!\!  D_i\tilde{\bf B}_y(u, v, w)+{2 \over F^2}D_iD_0{\bf L}_y(u,v,w)y.
\end{eqnarray}
By (\ref{GD5}),  (\ref{GD7}), and ${\bf L}=0$, we obtain
\begin{equation}
D_i\tilde{\bf B}_y(u, v, w)=\frac{2}{3}\Big\{D_i{\bf H}_y(u,v) {\bf h}_y(w)+D_i{\bf H}_y(v, w) {\bf h}_y(u)+D_i{\bf H}_y(w, u) {\bf h}_y(v)\Big\}.\label{GD1}
\end{equation}
The relation \eqref{GD1} yields
\begin{eqnarray}
D_h\tilde{\bf B}_y(u, v, \partial_k)-D_k\tilde{\bf B}_y(u, v, \partial_h)&=&\frac{2}{3}\Big\{D_h{\bf H}_y(u,v) {\bf h}_y(\partial_k)-D_k{\bf H}_y(u,v) {\bf h}_y(\partial_h)\Big\}\nonumber\\
&+&\frac{2}{3}\Big\{\big(D_h{\bf H}_y(v, \partial_k)-D_k{\bf H}_y(v, \partial_h) \big) {\bf h}_y(u)\Big\}\nonumber\\
&+&\frac{2}{3}\Big\{\big(D_h{\bf H}_y(\partial_k, u)-D_k{\bf H}_y(\partial_h, u)\big) {\bf h}_y(v)\Big\}.\label{GD1b}
\end{eqnarray}
By definition, we have $tr(\tilde{\bf B})=2{\bf H}$ and $tr({\bf h})=1$.
Then, (\ref{GD1b}) implies that
\begin{eqnarray}
D_h{\bf H}_y(u,\partial_k)-D_k{\bf H}_y(u,\partial_h)=2\Big\{D_h{\bf H}_y(u,\partial_k)-D_k{\bf H}_y(u,\partial_h)\Big\},\label{GD11}
\end{eqnarray}
which yields
\begin{equation}
D_h{\bf H}_y(u,\partial_k)=D_k{\bf H}_y(u,\partial_h).\label{GD12}
\end{equation}
Contracting (\ref{GD12}) with $y^h$ and using $D_k{\bf H}_y(u,y)=0$, we get
\begin{equation}
D_0{\bf H}_y(u,w)=0.\label{GD13}
\end{equation}
Take an arbitrary unit vector $y\in T_xM$ and an arbitrary vector $v\in T_xM$. Let $c(t)$ be the geodesic with $\dot c(0)=y$ and $U=U(t)$ the parallel vector field along $c$ with $V(0)=v$. In order to avoid clutter, we put
\begin{equation}
{\bf E}(t)={\bf E}_{\dot c}(U(t),U(t)), \ \ \ \ \ {\bf H}(t)={\bf H}_{\dot c}(U(t),U(t)).\label{GD14}
\end{equation}
From the definition of ${\bf H}_y$, we have
\begin{equation}
{\bf H}(t)={\bf E}^{'}(t).\label{GD15x}
\end{equation}
By  (\ref{GD13}) we have ${\bf H}^{'}(t)=0$  which implies that
\begin{equation}
{\bf H}(t)={\bf H}(0).\label{GD15xx}
\end{equation}
Then by (\ref{GD15x}) and (\ref{GD15xx}), we get
\begin{equation}
{\bf E}(t)={\bf H}(0)t+{\bf E}(0).\label{GD10}
\end{equation}
Since ${\bf E}(t)$ is a bounded function on $[0,\infty)$, then letting $ t\rightarrow +\infty $ or $ t\rightarrow -\infty $ implies that
\[
{\bf H}_y(v,v)={\bf H}(0)=0.
\]
Therefore ${\bf H}=0$. According to Akbar-Zadeh's theorem every  Finsler metric  $F=F(x, y)$ of scalar flag curvature ${\bf K}={\bf K}(x, y)$
on an $n$-dimensional  manifold $M$ has isotropic flag curvature ${\bf K}={\bf K}(x)$  if and only if $\textbf{H}=0$ \cite{AZ}. Every Finsler surface has scalar flag curvature ${\bf K}={\bf K}(x, y)$. Then by Akbar-Zadeh theorem,  we get ${\bf K}={\bf K}(x)$.
\qed

\bigskip

\bigskip

Now, we can prove the Theorem \ref{MainTHM1}.

\bigskip

\noindent
{\bf Proof of Theorem \ref{MainTHM1}:} Let $(M, F)$ be a homogeneous Landsberg surface and fix a point $x\in M$. Suppose that $y=y(t)$ is a unit speed parametrization of indicatrix of $M$ at $x$. We know that the curvature along $y(t)$ is completely determined by the Cartan scalar of $F$, i.e., we have
\[
{\bf K}(t)={\bf K}(0)\ e^{\int_0^tI(s)ds}.
\]
Thus either ${\bf K}(t)$ vanishes every where or it is non-zero every where  and ${\bf K}(t)$ has the same sign as the sign of ${\bf K}(0)$. On the other hand, for homogeneous Finsler surfaces the flag curvature is a bounded scalar function on $SM$. Suppose that $\lambda_1\leq {\bf K}(t)\leq \lambda_2$. In this case, we have

\[
e^{\lambda_1 t}\leq C(0)e^{\int_0^t {\bf K}(s)ds}\leq e^{\lambda_2 t}.
\]
Suppose that $C(0)\neq 0$. Then we consider two following cases:\\\\
{\bf Case 1:}  If $\lambda_1$ and $\lambda_2$ are positive, then letting $t\to \infty$ implies that $C(t)$ is unbounded, which is a contradiction.\\\\
{\bf Case 2:} If $\lambda_1$ and $\lambda_2$ are negative, then letting $t\to -\infty$ implies that $C(t)$ is unbounded, which is a contradiction.\\\\
Thus, every homogeneous Landsberg surface is Riemannian or flat. On the other hand, by Akbar-Zadeh's theorem   any positively complete Finsler metric with zero flag curvature  must be locally Minkowskian if the first and second Cartan torsions are bounded \cite{AZ}. For the homogeneous Finsler metrics, the first and second Cartan torsions are bounded. Then in this case, $F$ reduces to a locally Minkowskian metric. This completes the proof.
\qed

\bigskip

It is worth to mention that, in general, every Landsberg metric of non-zero scalar flag curvature is Riemannian, provided that its dimension is greater than two. Theorem \ref{MainTHM1} is Numata type theorem for homogeneous Finsler surfaces.

\begin{cor}
Let $(M, F)$ be a  homogeneous Finsler surface of non-positive flag curvature. Then $F$ is a Landsberg metric if and only if it has isotropic flag curvature. In this case, $F$ is Riemannian or locally Minkowskian.
\end{cor}
\begin{proof}
According to Theorem 8.1 of \cite{BCS}, every geodesically complete Finsler surface of non-positive isotropic flag curvature ${\bf K}(x)\leq 0$ and bounded Cartan scalar is a Landsberg metric.  Then,  by Theorem \ref{MainTHM1}, we get the proof.
\end{proof}

\bigskip

In \cite{1924}, L. Berwald introduced a non-Riemannian curvature so-called stretch curvature and denoted it by ${\bf \Sigma}_y$. He  showed that this tensor vanishes if and only if the length of a vector remains unchanged under the parallel displacement along an infinitesimal parallelogram. 
\begin{cor}
Every homogeneous stretch surface is Riemannian or locally Minkowskian.
\end{cor}
\begin{proof}
Every Landsberg metric is a stretch metric. In \cite{TN1}, it is proved that every homogeneous stretch metric is a Landsberg metric. Then,  by Theorem \ref{MainTHM1}, we get the proof.
\end{proof}

\bigskip

In \cite{BF}, Bajancu-Farran introduced a new class of Finsler metrics, called generalized Landsberg metrics. This class of Finsler metrics contains the class of Landsberg metrics as a special case. A Finsler metric $F$ on a manifold $M$ is called  generalized Landsberg metric the Riemannian curvature tensors of the Berwald and Chern connections coincide.  
\begin{cor}
Every homogeneous generalized Landsberg surface is  Riemannian or locally Minkowskian.
\end{cor}
\begin{proof}
By definition, we have
\be
L^i_{\ jl|k}-L^i_{\ jk|l}+L^i_{\ sk}L^s_{\ jl}-L^i_{\ sl}L^s_{\ jk}=0,\label{GL}
\ee
where ``$|$" denotes the horizontal derivation with respect to the Berwald connection of $F$. By \eqref{GL}, we get
\begin{eqnarray}
&&L_{isk}L^s_{\ jl}-L_{isl}L^s_{\ jk}=0,\label{GL4} \\
&&L_{ijl|k}-L_{ijk|l}=0.\label{GL5}
\end{eqnarray}
The Landsberg curvature of Finsler surface satisfies
\be\label{B3b}
L_{jkl}+\mu FC_{jkl}=0.
\ee
where $\mu:=-{4I_{,1}}/{I}$.  By \eqref{GL4} and \eqref{B3b}, we get
\begin{eqnarray}
\mu F\Big\{C_{isk}C^s_{\ jl}-C_{isl}C^s_{\ jk}\Big\}=0.\label{GL5}
\end{eqnarray}
We have two cases: If $C_{isk}C^s_{\ jl}-C_{isl}C^s_{\ jk}=0$, then the vv-curvature is vanishing. In \cite{Sc}, Schneider proved that vv-curvature is vanishing if and only if $F$ is Riemannian.   If $\mu=0$, then by \eqref{B3b} it follows that $F$ is a Landsberg metric.  By Theorem \ref{MainTHM1}, we get the proof.
\end{proof}

\bigskip
Let us define ${\bf \tilde J}= {\tilde J}_{ij}dx^i\otimes dx^j$, by 
\be
{\bf \tilde J}:=\big(J_{i,j}+J_{j,i}\big)_{|m}y^m.\label{X1}
\ee
In \cite{X}, Xia proved that every $n$-dimensional compact Finsler manifold with ${\bf \tilde J}=2{\bf \tilde H}$ is a weakly Landsberg metric. Here, we prove the following.
\begin{cor}\label{Prop1}
Every homogeneous Finsler surface satisfying  ${\bf \tilde J}=2{\bf \tilde H}$ is  Riemannian or locally Minkowskian.
\end{cor}
\begin{proof}
The following Bianchi idenity holds
\be
H_{ij}:=\frac{1}{2}\big(J_{i,j}+J_{j,i}-(I_{i,j})_{|p}y^p\big)_{|m}y^m.\label{X2}
\ee
See \cite{X}. By \eqref{X1} and \eqref{X2}, we get $(I_{i,j})_{|p}y^p=0$ and contracting it with $y^j$ yields
\be
J_{i|p}y^p=0.\label{X3}
\ee
For any geodesic $c=c(t)$ and any parallel vector field $U=U(t)$ along $c$, let us put
\[
{\bf I}(t)={\bf I}_{\dot c}\big(U(t),U(t), U(t)\big),  \ \ \ \  {\bf J}(t)={\bf J}_{\dot c}\big(U(t),U(t), U(t)\big).
\]
Thus, we have
\begin{equation}
{\bf J}(t)={\bf I}^{'}(t).\label{GD15}
\end{equation}
Integrating \eqref{X3} implies that
\[
{\bf I}(t)={\bf J}(0)t+{\bf I}(0).
\]
Every homogeneous manifold $ M $ is complete and the parameter $t$ takes all the values in $ (-\infty,+\infty) $. Letting $ t\rightarrow +\infty $ or $ t\rightarrow -\infty $ we have $ \textbf{I}(t) $ is unbounded which is a contradiction. Therefore $ \textbf{J}(0)={\bf J}(t)=0 $. On the other hand, every Finsler surface is C-reducible
\be
{\bf C}_y(u, v, w)= {1\over 3}\Big\{{\bf I}_y(u){\bf h}_y(v, w)+{\bf I}_y(v){\bf h}_y(u, w)+{\bf I}_y(w){\bf h}_y(u, v) \Big\}.\label{CR}
\ee
Taking a horizontal derivation of \eqref{CR} yields
\be
{\bf L}_y(u, v, w)= {1\over 3}\Big\{{\bf J}_y(u){\bf h}_y(v, w)+{\bf J}_y(v){\bf h}_y(u, w)+{\bf J}_y(w){\bf h}_y(u, v) \Big\}.\label{LR}
\ee
Putting ${\bf J}=0$ in  \eqref{LR} implies that ${\bf L}=0$.  By Theorem \ref{MainTHM1}, we get the proof.
\end{proof}

\bigskip

\noindent
Akbar Tayebi\\
Department of Mathematics, Faculty of Science\\
University of Qom \\
Qom. Iran\\
Email:\ akbar.tayebi@gmail.com

\bigskip

\noindent
Behzad Najafi\\
Department of Mathematics and Computer Sciences\\
Amirkabir University (Tehran Polytechnic)\\
Hafez Ave.\\
Tehran. Iran\\
Email:\ behzad.najafi@aut.ac.ir

\end{document}